\documentclass[11pt]{elsarticle}
\usepackage{subfig}
\usepackage{amsfonts}
\usepackage{amssymb}
\usepackage{graphicx,color}
\usepackage{epstopdf}
\usepackage{epsfig}
\usepackage[usenames,dvipsnames]{xcolor}
\usepackage{amsmath}
\usepackage{amsthm}
\newcommand{\be}{\begin{equation}}
	\newcommand{\ee}{\end{equation}}

\newtheorem{thm}{Theorem}

\newtheorem{cor}[thm]{Corollary}
\newtheorem{prop}[thm]{Proposition}

\theoremstyle{remark}

\theoremstyle{definition}

\begin{document}
	
	\begin{frontmatter}
		
		\author[FortWayne]{Peter D. Dragnev \fnref{Dragnev}}
		\fntext[Dragnev]{The research of this author was supported in part by NSF grant DMS-1936543.}
		\ead{dragnevp@pfw.edu}
		\author[FortWayne]{Alan R. Legg \corref{Legg}}
		\cortext[Legg]{Corresponding Author}
		\ead{leggar01@pfw.edu}
		\author[Vanderbilt]{Edward B. Saff}
		\ead{edward.b.saff@vanderbilt.edu}
		\address[FortWayne]{Purdue University Fort Wayne, Fort Wayne, IN 46805 (USA)}
		\address[Vanderbilt]{Vanderbilt University, Nashville, TN 37240 (USA)}
		\date{\today}
		\title{Point Source Equilibrium Problems with Connections to Weighted Quadrature Domains}
		\begin{abstract}
				We explore the connection between supports of equilibrium measures and quadrature identities, especially in the case of point sources added to the external field $Q(z)=|z|^{2p}$ with $p \in \mathbb{N}$.  Along the way, we describe some quadrature domains with respect to weighted area measure $|z|^{2p}dA_z$ and complex boundary measure $|z|^{-2p}dz$.
		\end{abstract}
	\end{frontmatter}

	\section{Introduction}
	The purpose of this article is to analyze the connections between the supports of equilibrium measures in the plane $\mathbb{C}$, and domains that exhibit quadrature identities for certain weighted area and complex boundary measures. We will review the notions of equilibrium measure and quadrature domain and then discuss how they relate in general. This will serve to motivate a search for domains that exhibit particular weighted quadrature identities.The result will be a description of the supports of equilibrium measures in the context of external fields with additional point sources.  For a discussion of several varieties of quadrature domains, see \cite{KhavinsonLundberg}.
	
	\subsection{External Fields and Equilibrium Measures} 
	Consider a unit charge placed onto the complex plane that is free to distribute into the configuration of least logarithmic energy under the influence of an external field. The external field is given as an extended real-valued function and the unit charge distribution placed into the plane is conceived as a probability measure.
	
	To be precise, suppose an external field $Q(z): \mathbb{C} \to \mathbb{R} \cup \{\infty\}$ is given which is {\it admissible}. This means $\exp(-Q(z))$ is upper-semicontinuous, positive-valued on a set of positive logarithmic capacity, and satisfies $|z|\exp(-Q(z))\to 0$ as $|z| \to \infty$ (see \cite{ST}). For any probability measure $\mu$ supported in the plane, the {\it weighted logarithmic energy} of $\mu$ is 
	\[\int_\mathbb{C} \int_\mathbb{C} \ln \frac{1}{|w-z|}d\mu(w)d\mu(z)+2 \int_\mathbb{C} Q(w)d\mu(w). \]

	In the presence of such an admissible external field, there is a unique energy-minimizing probability measure $\mu_Q$ called the {\it equilibrium measure} of the system. It can be described in terms of its logarithmic potential by the Frostman Theorem (see \cite{ST} Theorem I.1.1.3). A consequence is that the equilibrium measure has constant {\it weighted potential} on its support (quasi-everywhere\footnote{A property holds {\it quasi-everywhere} if the set of points where it does not hold has capacity zero. A set has capacity zero when $\int \int \ln |x-y|^{-1}d\mu(x)d\mu(y)=+\infty$ for every probability measure $\mu$ supported on the set.}). The physical interpretation is that a potential difference would induce a current which would redistribute the charge.  Mathematically, for some constant $C_Q$, and letting $S_Q=\text{supp} (\mu_Q)$ be the support of the equilibrium measure, the following holds for $z \in S_Q$ (quasi-everywhere):
	\[\int_\mathbb{C} \ln \frac{1}{|w-z|}d\mu_Q(w)+Q(z)=C_Q. \] 
	The weighted potential is also at least $C_Q$ outside $S_Q$. These conditions on the weighted potential in fact characterize the equilibrium measure.
	
	We are especially interested in the case of a smooth subharmonic admissible external field $Q(z)$ to which additional point sources are added, say at $z_1, z_2, \dots, z_n \, \in \mathbb{C}$ of positive intensities $q_1, q_2, \dots, q_n$, respectively. Setting $q:=\sum_{j=1}^nq_j,$ we consider the new external field \[V(z):=(1+q)Q(z)+\sum_{j=1}^nq_j\ln|z-z_j|^{-1}.\]
	With $S_Q=\text{supp}(\mu_Q)$ and $S_V=\text{supp}(\mu_V)$, we would like to characterize the equilibrium support $S_V$ and see how it compares to $S_Q$. 
	
	Adding the point sources causes an outward flux that tends to displace the support of the equilibrium measure. The factor $(1+q)$ in $V(z)$ is chosen to compensate for this with a greater inward influence from infinity.  (We will see that this choice is natural insofar as $S_V$ will take the simple form $S_Q \backslash \Omega$ for an open set $\Omega$, when the point sources have low intensity and are placed inside $S_Q$.) Since {\it quadrature domains} will play an important role in our approach to these problems, we briefly review their essential features.

	\subsection{Quadrature Domains} A classical quadrature domain $\Omega$ for a given test class of functions is a domain in $\mathbb{C}$ such that for all functions $f$ in the test class, integration over the domain is equal to a linear combination of point evaluations of the function and its derivatives. In other words the following holds: 
	\begin{equation}
		\label{QI}
		\int_{\Omega}f(z)dA_z=\sum_{j=1}^N\sum_{k=0}^{n_j} c_{jk} f^{(k)}(z_j), \end{equation}
	where  $dA_z$ is Lebesgue area measure, the $z_j$ are points in $\Omega$ and the $c_{jk}$ are constants that do not depend on $f$. When the variable of integration is understood we may drop the subscript from the area measure and write $dA$. Formula (\ref{QI}) is known as a {\it quadrature identity}. Common choices of test class include the space $\mathcal{A}^1(\Omega)$ of integrable analytic functions, the space of integrable analytic functions with primitive, harmonic functions, the Hardy space $\mathcal{H}^2(\Omega)$, and the Bergman space $\mathcal{A}^2(\Omega)$ of square-integrable analytic functions. For more on the foundations of quadrature domains, see \cite{AS, Avci, Bell4, BellBook, Gustafsson, GS, Shapiro2, Proc} Here and throughout this article, `analytic' will be taken to mean complex analytic.
	
	Quadrature domains whose quadrature identity involves just a multiple of a single point evaluation have been dubbed `one-point' quadrature domains (see, for example, \cite{Bell5,HJ}). Usually the term implies that the point evaluation involves a function value, and not a derivative value. We will keep that convention here.
	
	In this article we admit the following generalization: integration on the left side of the quadrature identity will occur with respect to modified measures. We will especially be concerned with the complex measure $|z|^{-2p}dz$ along the boundary of the domain and with the positively weighted area measure $|z|^{2p}dA_z$, $p\in \mathbb{N}$. For simplicity our domains will always be $\mathcal{C}^\infty$ smooth and bounded. The default test class for quadrature identities shall be $\mathcal{A}^{\infty}(\Omega),$ the space of analytic functions on $\Omega$ that are $\mathcal{C}^\infty$ smooth up to the boundary. By density this test class will allow us to appeal to the Hardy and Bergman spaces when needed (e.g. \cite{BellBook} first paragraph of Chapter 4, Theorem 6.2, and Cor 15.1).
	
	The assumption of smoothness (and sometimes simple-connectedness) is made to keep the focus on an elegant connection between equilibrium and quadrature. Classical quadrature domains are always algebraic and admit only certain boundary singularities \cite{sakai, Gustafsson,AS}. On the other hand, general questions of `regularity' for equilibrium problems are nontrivial \cite{HM}. In Theorem \ref{cavity}, we see an example of a conclusion obtained under assumptions of smoothness that can be verified directly from Frostman's Theorem.
	
	\subsection{Outline}
	In Section \ref{eqandqd} we show the relationship between supports of equilibrium measures and quadrature identities in the setting of a smooth and subharmonic admissible external field with additional point sources.
	
	In Section \ref{modified} we identify one-point quadrature domains with respect to the measures we mentioned above.  This will occur in three stages. We first consider domains that exclude the origin. Then we examine domains containing the origin, considering separately those where the origin is and is not the quadrature node.
	
	In the final Section \ref{conclude} we apply our results to the case of the admissible external field $Q(z)=|z|^{2p}$, $p\in \mathbb{N}$, in the plane with additional point sources. We include a connection to quadrature properties of {\it Cassini ovals} pointed out in Exercise 22.3a of \cite{KhavinsonLundberg}. We also examine more generally the case of what will be termed a {\it cavity}.
	
	\section{Equilibrium and Quadrature Identities}
	\label{eqandqd}
	
	Suppose $Q(z)$ is a smooth subharmonic admissible external field in the complex plane. The equilibrium measure $\mu_Q$ in the plane in the presence of $Q(z)$ will have compact support $S_Q$ (see e.g. \cite{ST}, \cite{Ransford}, \cite{HM}).  On the interior of $S_Q$ the equilibrium measure has the form $(2 \pi)^{-1}\Delta Q dA$. 
	
	To proceed we assume $S_Q$ is the closure of a smooth bounded domain, and that $\mu_Q=(2 \pi)^{-1}\Delta Q dA\big{|}_{S_Q}$ (without singular part on the boundary).  Under these assumptions we can describe $S_Q$ in terms of a conformal mapping of its exterior. (For a detailed discussion of smoothness in equilibrium problems, see \cite{HM}.)
	
	Let $S_Q^c := \hat{\mathbb{C}} \backslash S_Q$ be the exterior of $S_Q$, and let $\alpha$ be an interior point of $S_Q$. Frostman's Theorem says that for $z \in S_Q$, 
	\[\frac{1}{2\pi}\int_{S_Q} \Delta Q(w) \ln \frac{1}{|w-z|}dA_w+Q(z)=C \]
	for some constant $C$. Indicating holomorphic differentiation by $\partial$ and antiholomorphic differentiation by $\bar{\partial}$, recall that $\Delta Q= 4 \partial \bar{\partial} Q$. So in the equality above, we differentiate in $z$ to obtain
	\[\frac{1}{\pi} \int_{S_Q} \frac{\partial \bar{\partial}Q(w)}{w-z}dA_w + \partial Q (z)=0, \quad \quad z \in \text{int} \, (S_Q).\]
	By the Cauchy-Green Formula (e.g. \cite{BellBook} Theorem 2.1), this means
	\[\frac{1}{2 \pi i} \int_{\partial S_Q} \frac{\partial Q(w)}{w-z}dw=0. \]
	Now change variables conformally via $t := \frac{1}{w-\alpha}$, $\zeta := \frac{1}{z-\alpha}$, and let $S_Q^*$ and $(S_Q^c)^*$ be the images of $S_Q$ and $S_Q^c$ respectively under this change:
	\[-\frac{1}{2 \pi i} \int_{\partial S_Q^*}\frac{ \partial Q(t^{-1}+\alpha)t \zeta}{\zeta-t} \frac{1}{t^2} dt=0. \]
	Since $S_Q$ and $S_Q^c$ have the same boundary but with opposite orientations, 
	we rewrite this as
	\[\int_{\partial (S_Q^c)^*} \frac{t^{-1} \partial Q(t^{-1}+\alpha)}{t-\zeta} dt=0, \quad \zeta \in \text{int}(S_Q^*) \]
	By Mergelyan's theorem (e.g. \cite{Rudin} Chapter 20 ), we can pass from this equality by uniform convergence to 
	\[\int_{\partial (S_Q^c)^*} t^{-1} \partial Q (t^{-1}+\alpha)h(t)dt=0 \]
	for all $h \in \mathcal{A}^\infty((S_Q^c)^*)$. 
	
	In the language of quadrature domains, we can summarize the above discussion as follows: mapping the exterior of $S_Q$ via the mapping $z \to t:=(z-\alpha)^{-1}$, the resulting $(S_Q^c)^*$ is a null quadrature domain with respect to the complex boundary measure $t^{-1}\partial Q(t^{-1}+\alpha)dt$. 
	
	The above can also be applied to certain cases when the external field $Q(z)$ is perturbed by point sources. Consider the external field $V(z)=(1+q)Q(z)+\sum_{j=1}^n q_j\ln|z-z_j|^{-1}$, where point sources are added at locations $z_j \in \mathbb{C}$ with intensities $q_j>0$. Assume that the equilibrium measure support $S_V$ in the plane in the presence of $V(z)$ is the closure of a smooth bounded simply connected domain. Following the reasoning above for $V(z)$ in place of $Q(z)$, the $\ln|z-z_j|^{-1}$ terms will differentiate into Cauchy kernels. These will produce point evaluations when integrated. Thus $(S_V^c)^*$ in this case will be a quadrature domain with respect to $t^{-1}\partial Q(t^{-1}+\alpha)dt$, with quadrature nodes at the points $(z_j-\alpha)^{-1}$. For illustration, a diagram of the mapping used above is found in Figure \ref{fig1} for the case of a single point source added to a field $Q(z)=|z|^{2p}$.
	
	\begin{figure}%
		\centering
		\subfloat[\centering Shaded: $S_V$]{{\includegraphics[width=3.5cm]{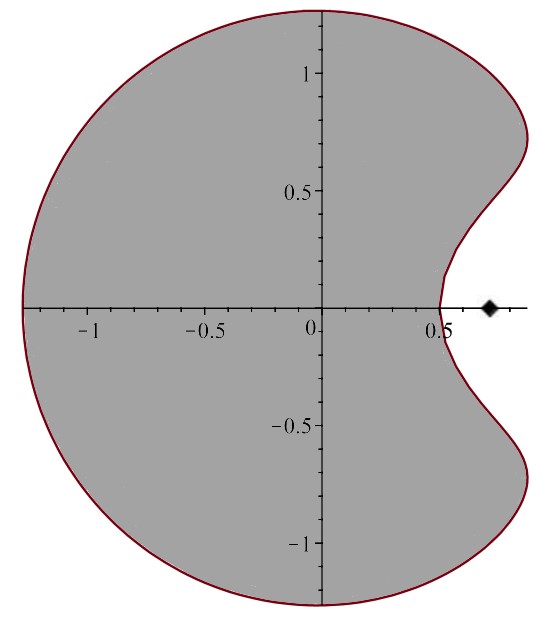} }}%
		\qquad
		\subfloat[\centering Unshaded: $(S_V^c)^*$]{{\includegraphics[width=5cm]{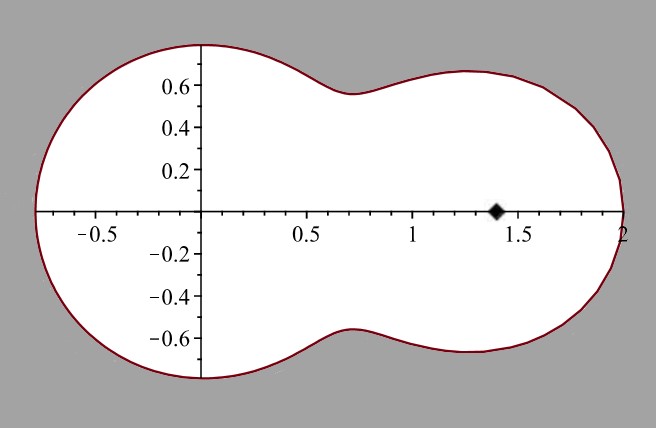} }}%
		\caption{ \\ {\bf Left:} Diagram of $S_V$ for $V(z)=C|z|^{2p}+q\ln|z-z_0|^{-1}$ when $S_V$ is simply connected. \\
			{\bf Right:} Conformally mapped exterior of support.}%
		\label{fig1}%
	\end{figure}
	
	
	
	
	On the other hand, when point sources of sufficiently small intensity and are placed into the interior of an existing equilibrium support, we expect by partial balayage \cite{GR, Roos} that the new equilibrium support will exclude neighborhoods of the point charges. (In other words, the point sources `sweep clean' a region of charge in their vicinity.) We will call these swept-clean voids in the support {\it cavities} (see Figure \ref{fig2}).  For examples see \cite{LD} where the cavities corresponding to point sources on the sphere are proven to come from quadrature domains; or \cite{BH}, where for external field $|z|^2$ the cavities are explained to be quadrature domains for Lebesgue area measure. On the sphere, quadrature domains have also appeared in \cite{CK1, CK2, CC} in the context of vector equilibrium problems and of fluid dynamics.  The effect of monomial terms added to a background external field was studied in \cite{KL}.
	
	\begin{figure}[h]
		\begin{center}
		\includegraphics[scale=0.4]{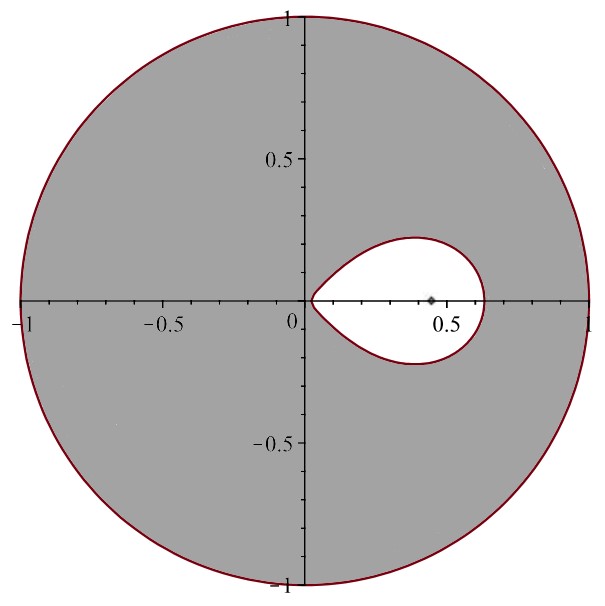}
		\end{center}
		\caption{\\ A cavity in $S_V$ for $V(z)=\frac{1+q}{4}|z|^4+q\ln|z-z_0|^{-1}$,
			\\ $q=.1995, z_0=\sqrt{0.2}$. }
		\label{fig2}
	\end{figure}
	
	Given the admissible external field $Q(z)$, let again $S_Q=\text{supp}(\mu_Q)$ be the support of the equilibrium measure in the plane in the presence of $Q(z)$.  Then add point sources of intensity $q_j>0$, $j=1,2,\cdots, n$ at points $z_1, z_2, \cdots, z_n \in \mathbb{C}$ respectively, and set $q=\sum_{j=1}^n q_j$. Consider the external field $V(z)=(1+q)Q(z)+\sum_{j=1}^nq_j \ln|z-z_j|^{-1}.$ If the $z_j$ are interior points of $S_Q$ and if the $q_j$ are sufficiently small, then by balayage the support $S_V$ of the equilibrium measure in the plane in the presence of $V(z)$ will be $S_Q \backslash \Omega$, for $\Omega$ an open set. In fact $\frac{1+q}{2 \pi}\int_\Omega \Delta Q(z) dA_z=q$ and $\frac{1+q}{2\pi}\int_{S_Q} \Delta Q(z) dA_z=(1+q)$.  In other words $S_V$ is formed by spreading a total charge of $(1+q)$ over the set $S_Q$, then deleting an open set amounting to a charge of $q$. In this setting we will call the components of $\Omega$ {\it cavities}. Recall that the factor $(1+q)$ in $V(z)$ is introduced so that the cavities are taken from $S_Q$. (Otherwise the point sources would tend to expand the outer boundary of the support in addition to forming cavities).
	
	In this situation, we can prove that $\Omega$ (the union of all cavities) is a `quadrature open set,' namely an open set that has a quadrature identity but which is not necessarily connected.
	
	By Frostman's condition for equilibrium measure $\mu_Q$,
	\[\frac{1+q}{2\pi} \int_{S_Q} \Delta Q(w)\ln \frac{1}{|w-z|}dA_w+(1+q)Q(z)=(1+q)C_Q \]
	for a constant $C_Q$, valid for $z \in S_Q$.
	Likewise Frostman's theorem applied to the equilibrium measure $\mu_V$ in the presence of $V(z):=(1+q)Q(z)+\sum_{j=1}^n q_j \ln|z-z_j|^{-1}$ yields
	\[\frac{1+q}{2\pi} \int_{S_V} \Delta Q(w) \ln \frac{1}{|w-z|}dA_w+(1+q)Q(z)+\sum_{j=1}^nq_j \ln \frac{1}{|z-z_j|}=C_V \]
	for a constant $C_V$, valid for $z \in S_V$.
	
	Differentiate these equalities with respect to $z$ and subtract one from the other. We see for $z \in S_V$ that
	\[\frac{1+q}{\pi}\int_\Omega \frac{\partial \bar{\partial} Q(w)}{w-z}dA_w-\sum_{j=1}^n\frac{q_j}{z_j-z}=0.\]
	In other words
	\[\int_\Omega \frac{\partial \bar{\partial} Q(w)}{w-z}dA_w=\frac{\pi}{1+q}\sum_{j=1}^n \frac{q_j}{z_j-z}. \] 
	
	Using an approximation theorem of Bers \cite{Bers}, we see by taking limits that
	\[\int_\Omega \Delta Q(w) h(w)dA_w=\frac{4\pi}{1+q}\sum_{j=1}^n q_j h(z_j), \]
	for all $h\in \mathcal{A}^1(\Omega)$, where $\mathcal{A}^1(\Omega)$ is the space of integrable holomorphic functions on $\Omega$. We conclude that $\Omega$ is a quadrature open set for weighted area measure $\Delta Q dA$.
	
	But $Q(z)$ is subharmonic, and the interior of $S_Q$ is in the region where $\Delta Q >0$. Thus if $\Omega$ had a component $\Omega_0$ that did not contain any of the $z_j$, then with test function $h(z)=1$ on that component and $h(z)=0$ on all others, we would have $\int_{\Omega_0} \Delta Q(z)dA_z=0$, which contradicts the positivity of $\Delta Q$. 
	
	So each component of $\Omega$ contains at least one of the $z_j,j=1,2,\ldots, n$ and is in fact a quadrature domain using weighted measure $\Delta Q dA$, with quadrature nodes at locations of point charges. We also note that in this case there was no assumption of boundary regularity, though implicitly we have assumed that $\partial S_Q$ has Lebesgue measure $0$, which is true for smooth enough $Q(z)$ \cite{HM}.
	
	We summarize this as our first theorem.
	
	\begin{thm} 
		\label{quadid}
		Let $Q(z)$ be a smooth subharmonic admissible external field in the plane. Let $V(z):=(1+q)Q(z)+\sum_{j=1}^n q_j\ln|z-z_j|^{-1}$ and set $q=\sum_{j=1}^nq_j$. Let $S_Q$ be the support of the equilibrium measure $\mu_Q$ in the plane, and $S_V$ the support of the equilibrium measure $\mu_V$. We assume that $\mu_Q$ and $\mu_V$ do not have singular parts supported on the boundaries of $S_Q$ and $S_V$, respectively.
		
		{\bf (i)} If $z_1, z_2, \dots, z_n$ are interior points of $S_Q$ and if $q_1, q_2, \dots, q_n$ are small positive intensities, so that $S_V=S_Q \backslash \Omega$ for an open set $\Omega \subset S_Q$, then each component $\mathcal{O}$ of $\Omega$ is a quadrature domain for weighted area measure $\Delta Q dA$. The quadrature nodes are all the points among the $z_1, z_2, \dots, z_n$ that lie in $\mathcal{O}$.
		
		{\bf (ii)} Suppose $S_V$ is smooth and simply connected with $z_1, z_2, \dots, z_n$ in the exterior of $S_V$. Let $\alpha$ be an interior point of $S_V$.  Let $S_V^c:=\hat{\mathbb{C}} \backslash S_V$ be the exterior of $S_V$, and let $(S_V^c)^*$ be the image of $S_V^c$ under the conformal mapping $z \mapsto t:=(z-\alpha)^{-1}$.  Then $(S_V^c)^*$ is a quadrature domain with respect to boundary measure $t^{-1} \partial Q(t^{-1}+\alpha)dt$. The quadrature nodes are the points $(z_j-\alpha)^{-1}$, $j=1,2,\dots, n$.
	\end{thm}

	\section{Quadrature Domains for modified measures}
	\label{modified}
	
	In anticipation of describing some equilibrium supports as above using the particular external field $Q(z)=C|z|^{2p}$, $C>0$, we now identify quadrature domains for the corresponding measures mentioned in Theorem \ref{quadid}. In this case $\Delta Q(z)=4p^2|z|^{2p-2}$ and $\partial Q (z)=pz^{p-1}\bar{z}^p.$ So let us focus on quadrature domains with respect to the weighted area measure $|z|^{2p}dA$, and the complex boundary measure $|z|^{-2p}dz$. In this section we treat such domains in their own right. In the next section we will make the connection again to equilibrium measures.
	
	The first thing to note is that quadrature domains with respect to these measures are related by Green's theorem.
	
	\begin{thm} Let $p$ be a natural number, and $\Omega$ a smooth bounded domain in the plane, $0 \notin \partial \Omega$, such that for some constants $c_{jk}$, and points $z_j \in \Omega$, the quadrature identity
		\[\int_{\partial \Omega}\frac{f(z)}{|z|^{2p}}dz=\sum_{j=1}^N\sum_{k=0}^{n_j} c_{jk} f^{(k)}(z_j)   \]
		holds for $f \in \mathcal{H}^2({\Omega})$. Then $\Omega$ is a quadrature domain with test class $\mathcal{A}^2(\Omega)$ with respect to the weighted area measure $|z|^{2p-2}dA$. 
	\end{thm}
	
	\begin{proof}
		For $f\in \mathcal{A}^\infty{(\Omega)}$, rewrite the right hand side of the quadrature identity using the Cauchy formula. Then note that there exists a rational function $R$ on $\Omega$ which is smooth up to the boundary with poles located at the $z_j$ (with order $n_j$) such that
		\[\int_{\partial \Omega} \frac{f(z)}{|z|^{2p}}dz= \int_{\partial \Omega} R(z)f(z)dz \]
		for all such $f$.
		
		Such $f$ are dense in the Hardy space. The structure of the orthogonal complement of the Hardy space on a smooth bounded domain (see \cite{BellBook}) ensures there is an $H \in \mathcal{A}^\infty({\Omega})$ such that
		\[\frac{1}{|z|^{2p}}-R(z)=H(z), \quad \quad \quad z \in \partial \Omega. \]
		Thus $|z|^{2p}$ has the same boundary values as does a meromorphic function smooth up to the boundary of $\Omega$.  This means that $\int_{\partial \Omega} f(z)|z|^{2p}dz$ is equal to a linear combination of evaluations of $f$ and its derivatives (the same linear combination for all $f$) by the Residue theorem.  For brevity, let $\mathcal{L}: \mathcal{A}^\infty({\Omega}) \to \mathbb{C}$ denote this linear combination of evaluations, and
		\[\int_{\partial \Omega} f(z) |z|^{2p}dz=\mathcal{L}f. \] 
		By Green's theorem
		\[\int_\Omega f(z)z^p\bar{z}^{p-1}dA=\frac{1}{2i}\mathcal{L}f. \] From here we consider separately the cases where $0 \in \Omega$ and $0 \notin \overline{\Omega}$.
		
		First, suppose $0 \notin \overline{\Omega}$. Then
		\[\int_\Omega f(z) |z|^{2p-2}dA=\int_\Omega \frac{f(z)}{z}z^p\bar{z}^{p-1}dA=\frac{1}{2i}\mathcal{L} \frac{f(z)}{z}, \] which is a quadrature identity.
		
		In case $0 \in \Omega$ we have
		
		\[\int_\Omega (f(z)-f(0))|z|^{2p-2}dA=\int_{\Omega} \frac{f(z)-f(0)}{z} z^p \bar{z}^{p-1}dA= \frac{1}{2i} \mathcal{L} \frac{f(z)-f(0)}{z}.  \]
		Hence we see
		\[ \int_\Omega f(z)|z|^{2p-2}dA=kf(0)+\frac{1}{2i} \mathcal{L} \frac{f(z)-f(0)}{z}, \]
		where $k=\int_\Omega |z|^{2p-2}dA$. This is a quadrature identity (which may introduce evaluations of derivatives of $f$ at $z=0$.)

	\end{proof}
	
	\subsection{One-point quadrature domains whose closures exclude the origin}
	
	For unweighted planar Lebesgue measure $dA$, it is well known that the only one-point quadrature domains for integrable harmonic functions are discs \cite{ES}. With the measures we are now considering, other possibilities arise. 
	
	We will consider separately the cases where the origin is in the domain or outside the closure. By a $p$th {\it root of a disc}, we will mean a set $\{ \sqrt[p]{z}: \, |z-a|<r \}$ for a given complex number $a$ and a positive number $r$, $r<|a|$, and where $\sqrt[p]{z}$ denotes an analytic branch of $z^{1/p}$ defined on the disc $|z-a|<r$. Recall that unless stated otherwise, the test class for quadrature identities is $\mathcal{A}^\infty$.
	
	\begin{thm}
		\label{discroot}
		With $p$ a natural number, let $\Omega$, $0 \notin \overline{\Omega}$, be a smooth bounded one-point quadrature domain with respect to the weighted area measure $|z|^{2p-2}dA$. Then $\Omega$ is a $p$th root of a disc.
	\end{thm}
	
	\begin{proof}
		By hypothesis there exists a complex constant $c$ and a point $z_0\in \Omega$ such that for all $f \in \mathcal{A}^\infty({\Omega})$,
		\[\int_\Omega f(z)|z|^{2p-2}dA=cf(z_0). \] 
		Since the origin is excluded from the closure of the domain, $f(z)/z^{p-1}$ is analytic and so
		\[\int_\Omega f(z) \bar{z}^{p-1}dA = \int_\Omega \frac{f(z)}{z^{p-1}} |z|^{2p-2}dA= \frac{c}{z_0^{p-1}}f(z_0). \] Now use Green's theorem, rewrite the right hand side using the Cauchy formula, and collect the terms to obtain for a new constant $C$
		\[\int_{\partial \Omega} f(z)(\bar{z}^{p}-\frac{C}{z-z_0})dz=0. \]
		By orthogonality to the Hardy space, we conclude that there exists $H \in \mathcal{A}^\infty (\Omega)$ such that 
		\[\bar{z}^{p}=H(z)+\frac{C}{z-z_0}, \quad \quad z \in \partial \Omega.\] 
		On setting $G(z):=H(z)(z-z_0)+C$, we have
		\[\bar{z}^{p}=\frac{G(z)}{z-z_0}, \quad \quad z \in \partial \Omega\]
		Let $T(z):=\sum_{j=0}^{p-1} z^{p-1-j}z_0^j$ and observe
		\[\bar{z}^{p}-\bar{z_0}^{p}=\frac{G(z)T(z)}{z^{p}-z_0^{p}}-\bar{z_0}^{p}, \quad z \in \partial \Omega \]
		Clearing the denominator,
		\[|z^{p}-z_0^{p}|^2=h(z), \quad \quad z \in \partial \Omega \]
		where $h(z):=G(z)T(z)-\bar{z_0}^{p}(z^{p}-z_0^{p})$ is analytic and smooth up to the boundary of $\Omega$.
		
		From here, the imaginary part of $h$ is $0$ on the boundary of $\Omega$, and by the maximum principle $h$ must be a positive real constant throughout $\Omega$, say $h(z) = K^2>0$.
		Thus on the boundary of $\Omega$
		\[|z^{p}-z_0^{p}|=K. \]
		
		For each $z$ on the boundary of $\Omega$, $z^{p}$ is on the circle of radius $K$ about $z_0^{p}$. If $K>|z_0^{p}|$, then the union of all values of branches of the $p$th roots of the points on the circle would comprise the boundary a single bounded domain, which would contain the origin. Then $\Omega$ would be this very domain, which is impossible since $0 \notin \Omega$.  Similarly if $K=|z_0^{p}|$ then $0$ would be a boundary point of $\Omega$ which is impossible.
		
		Thus $K<|z_0^{p}|,$ and taking ${p}$th roots of points on the circle of radius $K$ about $z_0^{p}$ gives a union of closed smooth curves, one for each branch of the root. One of these must be the boundary of $\Omega$. In other words, $\Omega$ is a $p$th root of a disc.
	\end{proof}
	
	Remark: the cavity in Figure \ref{fig2} is a root of a disc. 
	
	The case of boundary measure $|z|^{-2p}dz$ can be reduced to the theorem which was just proved.
	
	\begin{cor}
		\label{bdiscroot}
		If $\Omega$ is a smooth bounded one-point quadrature domain with respect to the boundary measure $|z|^{-2p}dz$, and if the origin is outside the closure of $\Omega$, then $\Omega$ is a $p$th root of a disc.
	\end{cor}
	\begin{proof}
		For some constant $c$ and point $z_0 \in \Omega$ we have for all $f \in \mathcal{A}^\infty({\Omega})$
		\[\int_{\partial \Omega} \frac{f(z)}{|z|^{2p}}dz=cf(z_0). \]	
		Using the Cauchy formula on the right side and collecting terms to the left side, we have
		\[\int_{\partial \Omega} f(z)\, \big{(}\, \frac{1}{|z|^{2p}}-\frac{c}{2 \pi i (z-z_0)}  \, \big{)}dz=0. \]
		By orthogonality in the Hardy space, there exists $H \in \mathcal{A}^{\infty}( {\Omega})$ such that 
		\[\frac{1}{|z|^{2p}}= H(z)+\frac{c}{2 \pi i (z-z_0)}, \quad \quad z \in \partial \Omega.  \]
		By the Argument Principle, since the left side exhibits no winding of argument around the boundary of $\Omega$, the meromorphic function on the right side has as many roots as poles (namely one) in $\Omega$. So for some nonzero $\xi \neq z_0$ and nonvanishing function $g \in \mathcal{A}^\infty({\Omega})$,
		\[|z|^{2p}=\frac{2\pi i c (z-z_0)}{g(z)(z-\xi)}, \quad \quad z \in \partial \Omega. \] 
		Since $z=0$ is outside the closure of $\Omega$, we write
		\[z^{p-1}\bar{z}^p=\frac{2\pi i c (z-z_0)}{g(z)(z-\xi)z}, \quad \quad z \in \partial \Omega. \]
		By the Residue Theorem there is a constant $C$ such that for all $f \in \mathcal{A}^\infty({\Omega})$
		\[\int_{\partial \Omega} f(z)z^{p-1}\bar{z}^pdz=Cf(\xi), \quad \quad z \in \partial \Omega. \]
		Use Green's Theorem on the left hand side:
		\[\int_\Omega f(z)|z|^{2p-2}dA=2Cif(\xi). \]
		In other words, $\Omega$ is a one-point quadrature domain with respect to weighted area measure $|z|^{2p-2}dA$, and we appeal to Theorem \ref{discroot}.
	\end{proof}
	
	\subsection{Quadrature nodes at the origin.}	
	
	Now we consider one-point quadrature domains which contain the origin. In this case, if the origin is the quadrature node and the measure for quadrature is $|z|^{2p}dA$, then the domain must be a disc.
	
	\begin{thm}
		If $\Omega$ is a smooth bounded one-point quadrature domain with respect to weighted area measure $|z|^{2p}dA$ whose quadrature node is the origin, then $\Omega$ is a disc centered at the origin.
	\end{thm}
	\begin{proof}
		We have in this case for all $f$ analytic and smooth up to the boundary
		\[\int_\Omega f(z)|z|^{2p}dA=cf(0). \]
		As in the proofs above, use Green's theorem on the left to get an integral on the boundary, and rewrite the right side using the Cauchy formula, subtract, and use orthogonality in the Hardy space to conclude that for some $H$ holomorphic and smooth up to the boundary of $\Omega$
		\[z^p\bar{z}^{p+1}=H(z)+\frac{c}{2\pi i z}, \quad \quad z \in \partial \Omega. \]
		Just multiply by $z$ and get
		\[|z|^{2p+2}=zH(z)+\frac{c}{2 \pi i}, \quad \quad z \in \partial \Omega. \]
		The rest is clear: the imaginary part of the holomorphic function on the right must be everywhere $0$, so that $|z|^{2p+2}$ is a real positive constant on the boundary of $\Omega$.
	\end{proof}
	
	The corresponding case for boundary measure $|z|^{-2p}dz$ will reduce to the case of weighted area measure, but with quadrature node away from the origin.
	
	\begin{prop}
		\label{propo}
		Let $\Omega$ be a smooth bounded one-point quadrature domain with respect to boundary measure $|z|^{-2p}dz$, whose quadrature node is the origin. Then $\Omega$ is a one-point quadrature domain with respect to weighted area measure $|z|^{2p-2}dA$, with quadrature node away from the origin.
	\end{prop}
	\begin{proof} 
		The proof is very similar to that of Corollary \ref{bdiscroot}, with $0$ in place of $z_0$.  Arguing similarly, we arrive at
		\[ \frac{1}{|z|^{2p}}=\frac{h(z)}{z}, \quad \quad z \in \partial \Omega. \]
		By the argument principle, since the left side has no winding, the function $h$ has exactly one root, at a point $z_0 \neq 0$.  Thus for some nonvanishing analytic function $g$
		\[|z|^{2p}=\frac{z}{g(z)(z-z_0)}, \quad \quad z \in \partial \Omega, \]
		or
		\[z^{p-1}\bar{z}^p=\frac{1}{g(z)(z-z_0)}, \quad \quad z \in \partial \Omega. \]
		Now integrate along the boundary, use Green's Theorem on the left hand side and the Residue Theorem on the right side, to conclude that for a constant $C$
		\[\int_{\Omega} f(z)|z|^{2p-2}dA=Cf(z_0). \]
		
	\end{proof}
	
	\subsection{Domains containing $0$, with quadrature node away from $0$.}
	We turn to the case of a one-point quadrature domain with respect to $|z|^{2p}dA$ which contains the origin, but whose quadrature node is not the origin. For this case, under the assumption of simple-connectedness we will be able to describe the domain via a Riemann map. 
	
	So let $\Omega$ be a smooth bounded simply connected one-point quadrature domain with respect to $|z|^{2p}dA$. Assume $0 \in \Omega$, but the quadrature node is $z_0 \neq 0$.  For some constant $c$,
	\[\int_\Omega f(z) |z|^{2p}=cf(z_0) \]
	for all $f$ analytic and smooth up to the boundary. Notice that the function
	\[\frac{ f(z)-\sum_{j=0}^{p-1} \frac{f^{(j)}(0)}{j!}z^j}{z^p} \] is analytic and smooth up to the boundary. Thus
	\[\int_\Omega  (f(z)-\sum_{j=0}^{p-1} \frac{f^{(j)}(0)}{j!}z^j) \bar{z}^p dA=kf(z_0)-\sum_{j=0}^{p-1} c_j f^{(j)}(0), \]
	for some constants $k,c_0,c_1, \cdots, c_{p-1}$. But now move the integrations of the terms under the summation sign to the right side and conclude that for some (different) constants $\kappa,d_0, d_1, \cdots, d_{p-1}$
	\[\int_\Omega f(z)\bar{z}^{p}dA=\kappa f(z_0)+ \sum_{j=0}^{p-1}d_j f^{(j)}(0). \]
	
	Since this quadrature identity will hold in the Bergman space, and since the left hand side is the Bergman space pairing $\langle f(z), \, z^p \rangle,$ we have
	
	\begin{equation}
		\label{Bspan}
		w^p=\kappa K_\Omega (w,z_0) + \sum_{j=0}^{p-1} d_j \frac{\partial^j K_\Omega(w,z)}{\partial \bar{z}^j} \bigg{|}_{z=0},
	\end{equation}
	where $K_\Omega(w,z)$ is the Bergman kernel function of $\Omega$.
	
	Let $\varphi: \mathbb{D} \to \Omega$ be a conformal map which takes $0 \mapsto 0$. We use $t$ to denote points in $\mathbb{D}$. By the Bergman kernel transformation formula \cite{Bergman, BellBook} the above equality can be pulled back to the disc, in the form
	\[\varphi'(t)\varphi(t)^p=aK_\mathbb{D}(t, \zeta_0)+\sum_{j=0}^{p-1} b_j \frac{\partial^j K_\mathbb{D} (t, \zeta)}{\partial \overline{\zeta}^j} \bigg{|}_{\zeta=0}, \]
	for some constants $a, b_0, b_1, \cdots, b_{p-1}$. Since we know the Bergman kernel of the disc to be \[K_\mathbb{D}(t,\zeta)=\frac{1}{\pi(1-t\overline{\zeta})^2},\]
	we have now for some constants $\alpha, \beta_0, \beta_1, \cdots, \beta_{p-1}$
	\[(\varphi(t)^{p+1})'=\frac{\alpha}{(1-t \overline{\zeta}_0)^2}+\sum_{j=0}^{p-1} \beta_j t^j. \]
	That means
	\[\varphi(t)^{p+1}=\frac{P(t)}{1-t \overline{\zeta}_0} \]
	for some polynomial $P(t)$ of degree at most $p+1$ such that $P(0)=0$.
	Notice that the denominator has no roots in the disc, and the meromorphic function on the right must have an analytic $(p+1)$-th root. The origin must be a root of order exactly $p+1$, and consequently for some constant $C$ and appropriately chosen branch for the root, we have
	\begin{equation}
		\label{form}
		\varphi(t)=\frac{Ct}{(1-t \overline{\zeta}_0)^{1/{p+1}}}. \end{equation}
	\begin{thm}
		\label{theo} If $\Omega$ is a smooth bounded simply connected one-point quadrature domain for measure $|z|^{2p}dA$ whose quadrature node is not the origin, then $\Omega$ is the conformal image of the unit disc under a map of the form (\ref{form}) with $\zeta_0 \neq 0$. 
	\end{thm}
	{\bf Remarks.} In view of Proposition \ref{propo}, Theorem \ref{theo}  applies to smooth bounded simply connected one-point quadrature domains for boundary measure $|z|^{-2p-2}dz$ whose quadrature node is the origin. In Figure \ref{fig3} we see an example of such a quadrature domain.
	
	\begin{figure}[h]
		\begin{center}
		\includegraphics[scale=0.4]{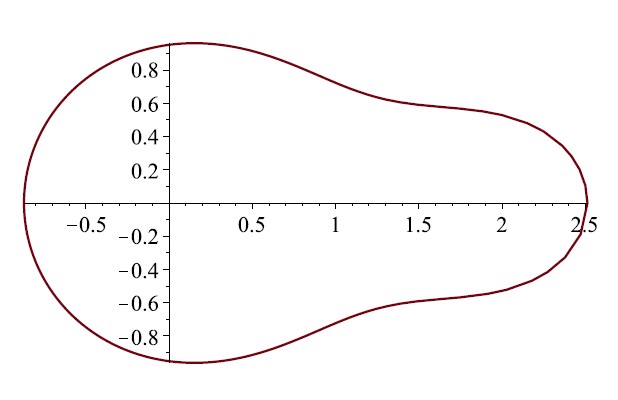}
		\end{center}
		\caption{A one-point quadrature domain for $|z|^{2}dA$ containing $0$, with node away from $0$.}
		\label{fig3}
	\end{figure}

	Lastly we come to smooth bounded one-point quadrature domains for the boundary measure $|z|^{-2p}dz$, with quadrature nodes away from the origin. It will again be possible to describe such domains using conformal mapping under the assumption of simple connectedness. 
	
	Let $\Omega$ be such a domain, for which
	\[\int_{\partial \Omega} \frac{f(z)}{|z|^{2p}}dz=cf(z_0), \]
	for some constant $c$ and point $0 \neq z_0 \in \Omega$, and for all $f \in \mathcal{A}^\infty({\Omega}).$ Using Cauchy's formula, orthogonality in the Hardy Space, and the Argument Principle in the way we have several times above, this means that for some nonvanishing $g \in \mathcal{A}^\infty({\Omega})$ and $\xi \in \Omega$, $\xi \neq z_0$,
	\[|z|^{2p}= \frac{z-z_0}{g(z)(z-\xi)}, \quad \quad z \in \partial \Omega.\]
	We now distinguish two cases, where $\xi \neq 0$ and $\xi=0$.
	
	In the first case, the boundary equality
	\[\bar{z}^p= \frac{z-z_0}{g(z)(z-\xi)z^p}, \quad \quad z \in \partial \Omega \]
	together with Green's Theorem gives
	\begin{equation}
		\label{evals}
		\int_\Omega f(z)\bar{z}^{p-1}dA=\alpha f(\xi)+ \sum_{j=0}^{p-1}\beta_j f^{(j)}(0)
	\end{equation}
	for some constants $\alpha, \beta_j, \, j=0,1,\cdots, p-1$. But now $z^{p-1}$ is a linear combination of Bergman kernel functions, similar to what occurred in \ref{Bspan}. Letting $\varphi(t)$ be a conformal map from the unit disc to $\Omega$ taking $0 \mapsto 0$, we have on the disc
	\[\varphi'(t) \varphi(t)^{p-1}=\frac{k}{(1-\omega \overline{\zeta}_0)}+\sum_{j=0}^{p-1}d_jt^j, \]
	where $k, d_j$ are constants.
	And so for a polynomial $P(t)$ of degree at most $p+1$ vanishing at the origin,
	\[\varphi(t)^p=\frac{P(t)}{1-t \overline{\zeta}_0}. \]
	For the right side to admit an analytic $p$th root, we need $P(t)=at^p+bt^{p+1}$ for some constants $a,b$ with $|a/b|>1$.
	And so we determine
	\begin{equation}
		\label{nonzero}
		\varphi(t)=t \, \big{(}\frac{a+bt}{1-t \bar{\zeta}_0} \big{)}^{1/p}.
	\end{equation}
	
	Now, consider the case where $\xi=0$. The argument begins the same way, using $\xi=0$;  but now instead of (\ref{evals}) we get
	\[\int_\Omega f(z)\bar{z}^{p-1}dA=\alpha f(\xi)+ \sum_{j=0}^{p}\beta_j f^{(j)}(0) \]
	for some constants $\alpha, \beta_j$ (note the upper index on the sum has changed). Now use the Bergman kernel with similar reasoning as above, to conclude that a conformal map from the unit disc to $\Omega$ taking $0 \mapsto 0$ will be of the form
	\begin{equation}
		\label{zero}
		\varphi(t)=t \, \big{(} \frac{a+bt+ct^2}{1-t \overline{\zeta}_0} \big{)}^{1/p}
	\end{equation}
	for some constants $a,b,c$ such that the roots of the quadratic function in the numerator have modulus greater than $1$.
	
	\begin{thm}
		\label{riemann} If $\Omega$ is a smooth bounded simply connected one-point quadrature domain for the boundary measure $|z|^{-2p}dz$, with $0 \in \Omega$  but $0$ not the quadrature node, then $\Omega$ is the conformal image of the unit disc under a map of form (\ref{nonzero}) or (\ref{zero}).
	\end{thm}

	\section{Applications}
	\label{conclude}
	
	\subsection{The case of \boldmath{$Q(z)=C|z|^{2p}$}}
	
	On the plane with external field $Q(z)=C|z|^{2p}$ with $C>0$, the support $S_Q$ of the equilibrium measure $\mu_Q$ is the closed disc $D_R(0)$ centered at the origin with radius $R:=(2pC)^{-1/2p}$ (see \cite{ST} Theorem IV.6.1).
	
	The equilibrium support $S_V$ in the plane in the presence of $V(z):=(1+q)Q(z)+q \ln|z-z_0|^{-1}$ will depend on $z_0$ and $q$. If $z_0$ is internal to $S_Q$ and if $q$ is small enough, then a cavity will form. On the other hand, if $q$ is large or $z_0$ is far from $S_Q$, then a cavity does not form.
	
	\begin{thm}
		
		Consider the external field $Q(z)=C|z|^{2p}$ in the plane, $C>0$, $p \in \mathbb{N}$. Let $V(z):=(1+q)Q(z)+q\ln|z-z_0|^{-1},$ where $q>0$ and $z_0 \in \mathbb{C}$. Let $S_V:=\text{supp}{(\mu_V)}$ and $S_Q:=\text{supp}(\mu_Q)$ be the supports of the equilibrium measures $\mu_V$, $\mu_Q$ respectively. Then $S_Q$ is the closed disc centered at the origin of radius $R=(2pC)^{-\frac{1}{2p}}$. We have 
		
		\[d\mu_Q=\frac{\Delta Q}{2\pi} dA\big{|}_{S_Q},\]
		\[d\mu_V=\frac{(1+q)\Delta Q}{2 \pi} dA \big{|}_{S_V}.\]
		The support $S_V$ is described as follows:
		
		{\bf (i)} If a cavity forms so that $S_V=S_Q \backslash \Omega$ for an open set $\Omega \subset S_Q$, and if $0 \notin \Omega$, then $\Omega$ is a $p$th root of a disc.
		
		{\bf (ii)} If a cavity forms so that $S_V=S_Q \backslash \Omega$ for a smooth open set $\Omega \subset S_Q$, where $0 \in \Omega$, then $\Omega$ is a conformal image of the unit disc under a map  $\varphi: \mathbb{D} \to \Omega$ of the form $\varphi(t)=\frac{At}{(1-t\overline{\zeta_0})^{1/p}},$ with $A$ a constant and $\zeta_0 \neq 0$ .
		
		{\bf (iii)} If $z_0=0$, then $S_V$ is an annulus.
		
		{\bf (iv)} If $S_V$ is smooth and simply connected, and if $0 \in S_V$ is an interior point, then let $S_V^c:=\hat{\mathbb{C}}\backslash S_V$ be the exterior of $S_V$. Let $(S_V^c)^*$ be the image of $S_V^c$ under the conformal mapping $z \mapsto {z^{-1}}$.  Then $(S_V^c)^*$ is the image of the unit disc under a conformal map $\varphi: \mathbb{D} \to (S_V^c)^*$ of the form $\varphi(t)=t \, \big{(} \frac{P(t)}{1-t \overline{\zeta_0}} \big{)} ^{1/p}$, where $P(t)$ is a polynomial of degree at most $2$.
		
	\end{thm}
	
	A few remarks are in order. Note that $\bf{(iii)}$ is a consequence of (\cite{ST} Section IV.6). In {\bf (iv)}, if $0 \notin \Omega$, then use instead the mapping $z  \to (z-\alpha)^{-1},$  where $\alpha$ is an interior point $S_V$, and analyze as in the previous section. Finally, we emphasize the assumptions of smoothness and simple connectedness. In the next subsection, we show how the case {\bf (i)} can be verified by direct computation.
	
	If instead of a single point source we add several to the background field $Q(z)$, then we can draw some conclusions about what happens in case a cavity forms. For the next theorem, consider the external field $Q(z)=C|z|^{2p}$ in the plane, $C>0$, $p \in \mathbb{N}$.  Let $V(z):=(1+q)Q(z)+\sum_{j=1}^n q_j \ln|z-z_j|^{-1}$ where $q=\sum_{j=1}^n q_j$. Here $q_j>0$, $j=1,2, \dots, n$, and $|z_j|<(2pC)^{-1/2p}$, $j=1,2,\dots, n.$ Set $r_j:=\sqrt{\frac{q_j}{2C(1+q_j)}}$ and define $\Omega_j$ to be that component of $\{z \in \mathbb{C}: \; |z^p-z_j^p|<r_j \}$ which contains $z_j$. (Note that $\Omega_j$ is a $p$th root of a disc.)
	
	\begin{thm}
		\label{multipoint}
		With definitions as above, suppose that $S_V$ is of the form $S_Q\backslash \Omega$ for an open set $\Omega \subset S_Q$.
		
		{\bf (i)} $\Omega$ is a union of quadrature domains for weighted area measure $|z|^{2p-2}dA$. The quadrature nodes involved are the points $z_j$.
		
		{\bf (ii)} If for each $j=1,2, \cdots, n$ we have $r_j<|z_j|^p$ and $|z_j^p|+r_j< (2pC)^{-1/2}$, and the $\Omega_j$ are mutually disjoint, then $\Omega=\cup_{j=1}^n \Omega_j.$ 
	\end{thm} 
	
	Part (i) is immediate from part (i) of Theorem \ref{quadid}. Part (ii) says that for sufficiently weak point sources in the interior of $S_Q$, the cavity is a union of $p$th roots of discs around the point sources. This can be verified by computing from Frostman's theorem similar to what will be shown in the next subsection.
	
	Khavinson and Lundberg in \cite{KhavinsonLundberg}, Exercise 22.3a, point out that Cassini ovals $\mathcal{O}_T:=\{z \in \mathbb{C} : \; |z^2-1|<T\}$, $T>1$, are two-node quadrature domains for weighted area measure $|z|^2dA$. One can show that such $\mathcal{O}_T$ are the cavities that occur in the support $S_V$ of the equilibrium measure in the presence of $V(z):=(1+2q)C|z|^4+q\ln|z-1|^{-1}+q\ln|z+1|^{-1}$, when the constant $C$ is small enough and $q$ is large enough. In slightly more generality, we have the following.
	
	{\bf Example.} Consider the external field $Q(z)=C|z|^{2p}$ in the plane, where $C<1/2p$. The corresponding equilibrium measure support $S_Q$ is the closed disc centered at the origin of radius $(2pC)^{-1/2p}.$ Let $\omega=\exp(2\pi i/p)$, and place point sources of equal intensity $q>0$ at the $p$th roots of unity. So consider the external field $V(z)=(1+pq)Q(z)+q\sum_{j=0}^{p-1} \ln |z-\omega^j|^{-1}.$ Denote $T:=T(q)=\sqrt{\frac{q}{2C(1+pq)}}$ and $\mathcal{O}_T:=\{z \in \mathbb{C}: \; |z^p-1|<T \}$. Then for all $q>0$ such that $1+T(q)<(2pC)^{-1/2}$, the equilibrium support $S_V$ is $S_Q \backslash \mathcal{O}_T$.  
	
	Notice that for $0<q<2C/(1-2pC),$  $\mathcal{O}_T$ is a union of $p$th roots of discs around the $p$th roots of unity. When $q=2C/(1-2pC)$, the closure of $\mathcal{O}_T$ becomes connected as the roots of discs meet at the origin.  For all values $q>2C/(1-2pC)$ such that $T(q)<(2pC)^{-1/2}-1$, the set $\mathcal{O}_T$ is a lemniscate which generalizes the Cassini ovals pointed out in \cite{KhavinsonLundberg} Exercise 22.3a. If $q$ is large enough so that $T(q)>(2pC)^{-1/2}-1$, then $S_V$ is no longer determined by a cavity (since $\mathcal{O}_T$ is no longer compactly contained in $S_Q$). Figure \ref{fig4} demonstrates the cases of $p=2,3$.
	
	This example is related to the fact that polynomial lemniscates are quadrature domains with respect to (unweighted) equilibrium measure on the boundary of the lemniscate. One interpretation of the cavity above is that for the lemniscate $\mathcal{O}_T$, the unweighted equilibrium measure on the boundary matches the balayage of the weighted area measure $\frac{ 1+q}{2\pi} \Delta QdA \big{|}_{\mathcal{O}_T}$ onto the boundary $\partial \mathcal{O}_T$. (See \cite{ST}, \cite{KhavinsonLundberg}.)
	
	
	
	
	
	\begin{figure}%
		\centering
		\subfloat[\centering $q$=0.095]{{\includegraphics[width=3.5cm]{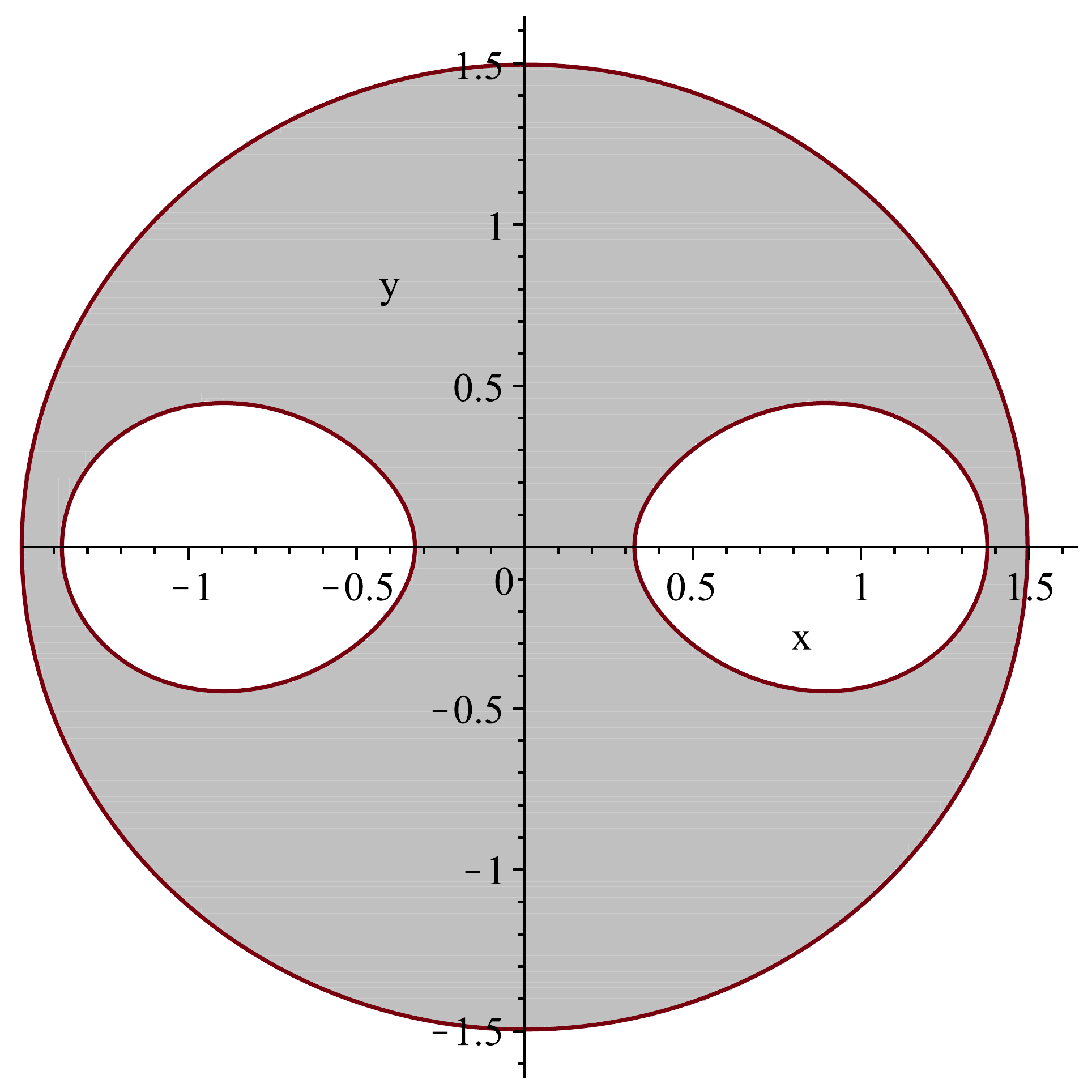} }}%
		\qquad
		\subfloat[\centering $q$=0.125]{{\includegraphics[width=3.5cm]{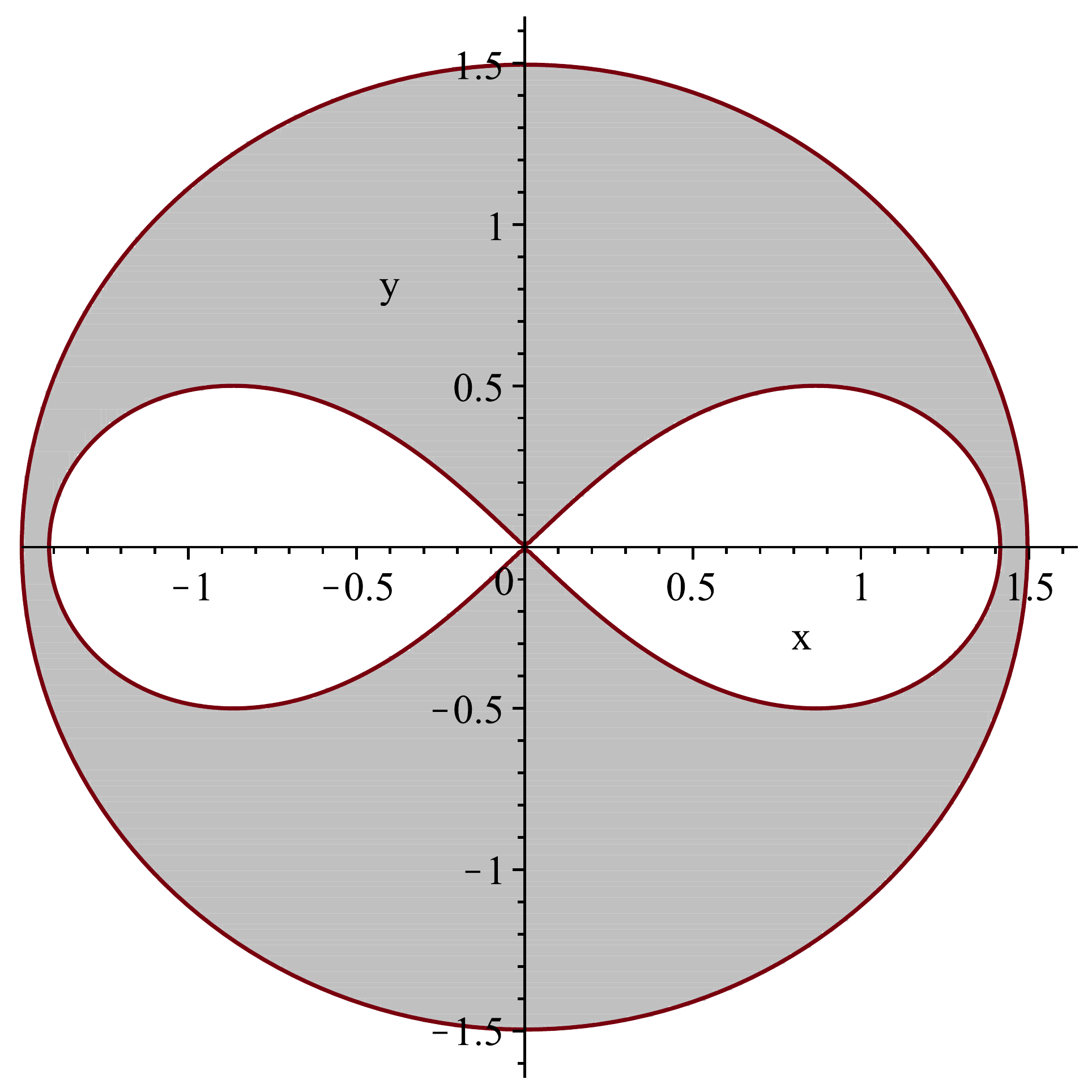} }}%
		\qquad
		\subfloat[\centering $q$=0.18]{{\includegraphics[width=3.5cm]{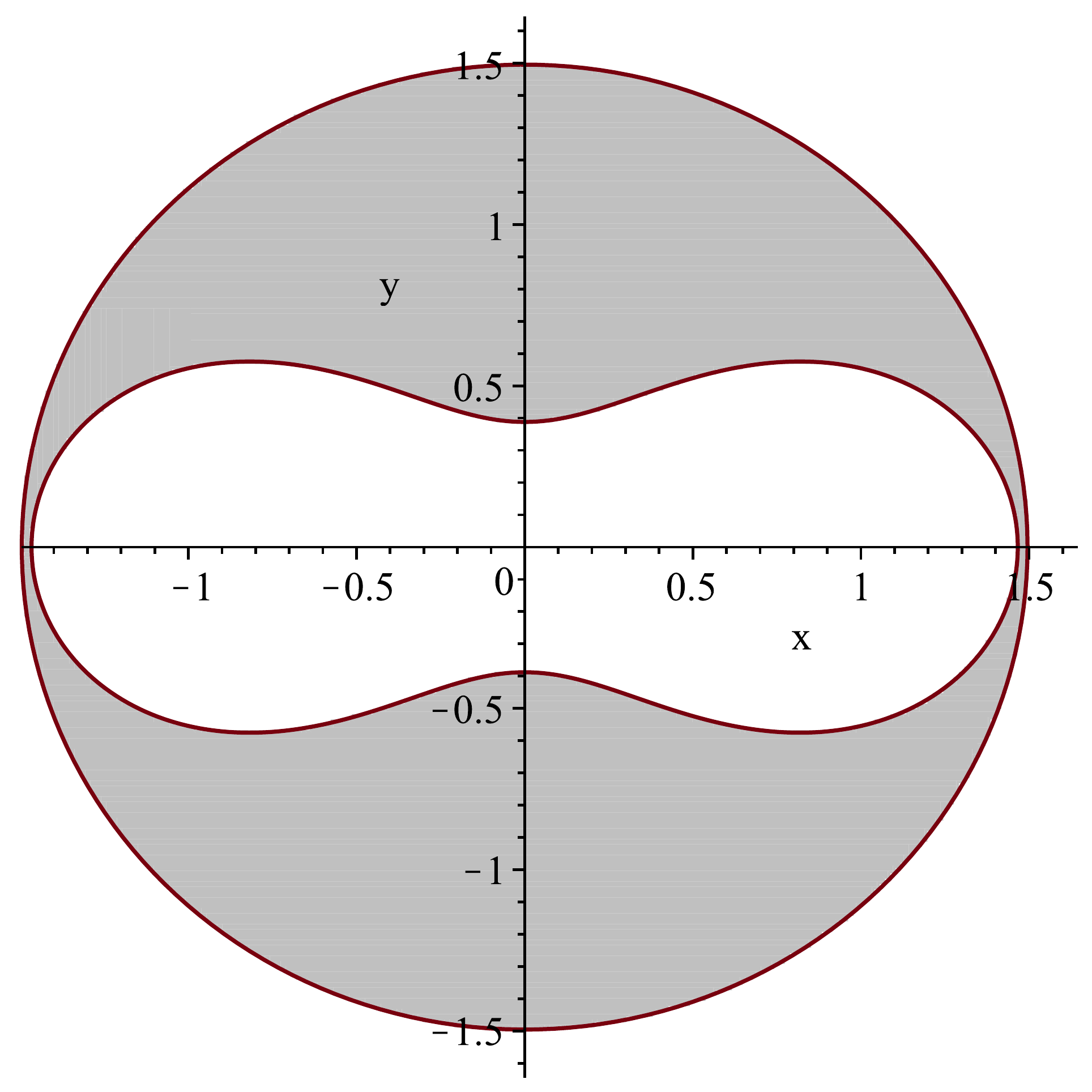} }}%
		\vspace{1mm}
		
		\subfloat[\centering $q$=0.0415]{{\includegraphics[width=3.5cm]{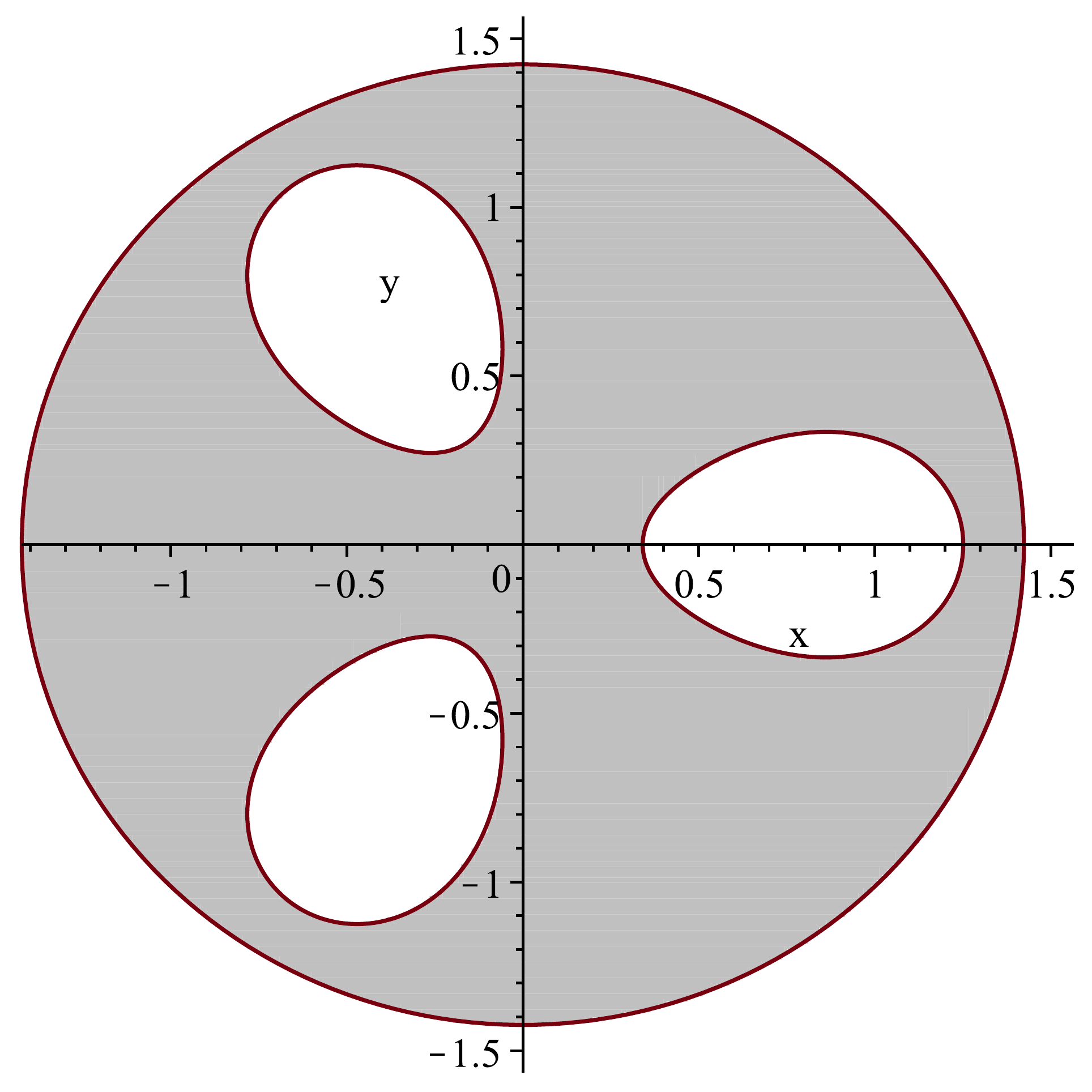} }}%
		\qquad
		\subfloat[\centering $q$=0.04545]{{\includegraphics[width=3.5cm]{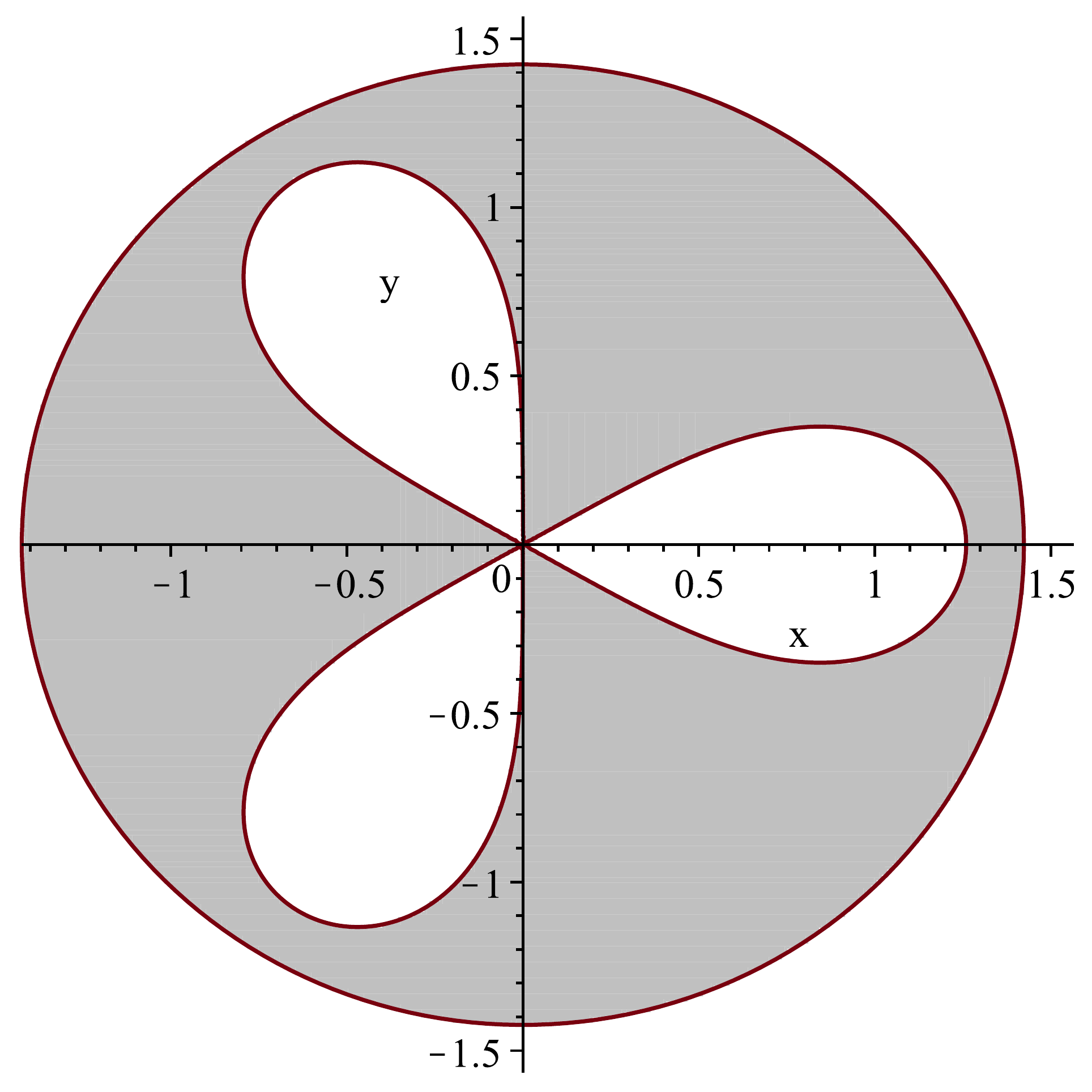} }}%
		\qquad
		\subfloat[\centering $q$=0.0475]{{\includegraphics[width=3.5cm]{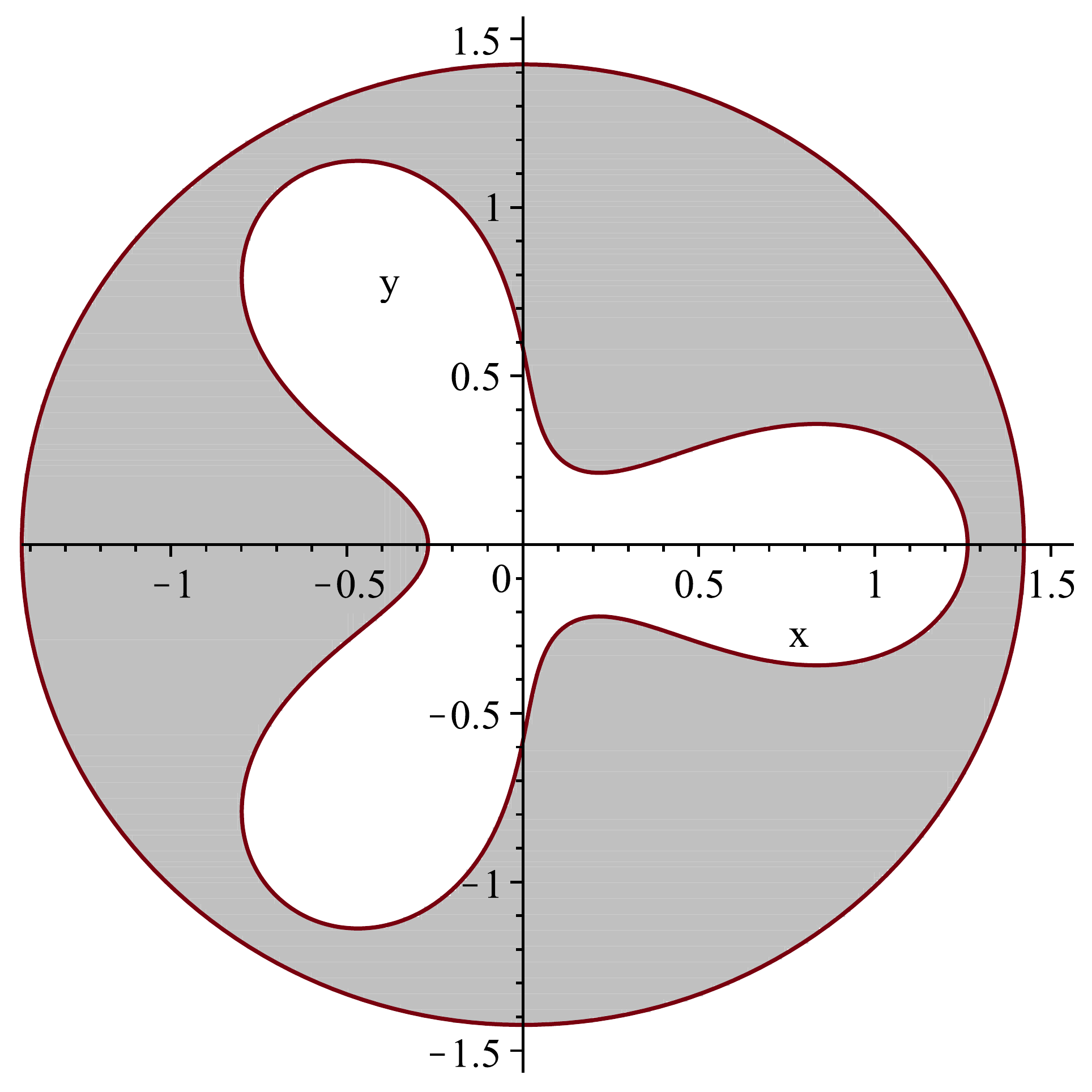} }}%
		
		\caption{ \\{\bf Above}: $S_V$ for $V(z)=\frac{1+2q}{50}|z|^{4}+q\ln|z^2-1|^{-1}$ with given $q$.\\{\bf Below}: $S_V$ for $V(z)=\frac{1+3q}{20}|z|^6+q\ln|z^3-1|^{-1}$ with given $q$.}   %
		\label{fig4}%
	\end{figure}

	\subsection{Cavity by direct computation}
	Although we have assumed smoothness above, in the case of a cavity our result can be verified by directly computing the weighted potential and using Frostman's theorem.  In fact this can be done in a more general setting which we now describe.
	
	Suppose that $Q(z)$ is an admissible external field, and that $\varphi(z)$ is a complex analytic function such that $Q(z)=C|\varphi(z)|^{2p}$ for $z$ near $z_0$, where $z_0$ is an interior point of the support $S_Q$ of the equilibrium measure $\mu_Q$. Assume furthermore that $0 \neq d \varphi^p/ dz |_{z=z_0}$. (Here $\varphi^p(z)$ designates the $p$th power, and not $p$ iterations of the function.)
	
	Note that $\varphi^p$ is univalent in a neighborhood of $z_0$, and so has a local inverse $\psi$ defined in a neighborhood of $\varphi^p(z_0)$. Let $r:=\sqrt{\frac{q}{2C(1+q)}}$, and let $q$ be small enough so that $\psi$ is defined on the disc $D_r(\phi^p(z_0))$ of radius $r$ centered at $\varphi^p(z_0)$, and so that $\Omega:=\psi(D_r(\varphi^p(z_0)))$ is compactly contained in $S_Q$.  We contend that for small $q>0$, the equilibrium measure $\mu_V$ in the plane in the presence of the external field $V(z)=(1+q)Q(z)+q\ln|z-z_0|^{-1}$ has support $S_V:=S_Q \backslash \Omega$.
	
	To proceed, define a weighted potential $F(z)$ as in Frostman's Theorem:
	
	\[F(z):= \frac{1+q}{2\pi}\int_{S_V} \Delta Q(w) \ln \frac{1}{|w-z|}dA_w + (1+q)Q(z)+ q \ln \frac{1}{|z-z_0|}. \]
	
	Write $F(z)=F_1(z)+F_2(z)$ where $F_1,F_2$ are defined as
	
	\[F_1(z):=\frac{1+q}{2\pi} \int_{S_Q} \Delta Q(w) \ln \frac{1}{|w-z|}dA_w + (1+q)Q(z), \]
	
	\[F_2(z):=q\ln \frac{1}{|z-z_0|} - \frac{(1+q)C}{2\pi} \int_\Omega 4p^2 |\varphi^{p-1}(w)\varphi'(w)|^2\ln \frac{1}{|w-z|}dA_w \]
	
	By Frostman's theorem, $F_1(z)$ is a constant $C_1$ for $z\in S_V$ and $F_1(z) \geq C_1$ elsewhere (quasi-everywhere). As for $F_2$, we need to compute the potential occurring from $\Omega$.  Observe that $p^2|\varphi^{p-1}(w)\varphi'(w)|^2$ is the Jacobian determinant of the map $w \mapsto \varphi^p(w)$. By changing variables the integral over $\Omega$ occurring in $F_2$ is
	\[\frac{4C(1+q)}{2\pi} \int_{D_r(\varphi^p(z_0))} \ln\frac{1}{|\psi(t)-z|}dA_t. \]
	When $z$ is outside $\Omega$, this can be evaluated by the harmonic mean value property, yielding $q\ln|z-z_0|^{-1}$. When $z$ is in the closure of $\Omega$, we rewrite the integral as
	\[\frac{4C(1+q)}{2\pi} \big{(} \int_{D_r(\varphi^p(z_0))} \ln \frac{|t-\varphi^p(z)|}{|\psi(t)-z|}dA_t + \int_{D_r(\varphi^p(z_0))} \ln \frac{1}{|t-\varphi^p(z)|}dA_t\big{)} . \]
	The first term can be calculated by the mean value theorem. The second term represents the potential formed by a disc of uniform charge, which is well known. This gives
	\[q \ln \frac{|\varphi^p(z_0)-\varphi^p(z)|}{|z_0-z|}+q(\frac{1}{2}-\ln(r))-\frac{q}{2} \frac{|\varphi^p(z_0)-\varphi^p(z)|^2}{r^2}. \]
	
	Hence
	\[F_2(z)= \begin{cases} 0, \qquad \qquad \qquad \qquad \qquad \qquad \qquad \qquad \; \; \,  z \notin \Omega \\
		q \ln \frac{r}{|\varphi^p(z_0)-\varphi^p(z)|}+\frac{q}{2}(-1+\frac{|\varphi^p(z_0)-\varphi^p(z)|^2}{r^2}), \; z \in \Omega \end{cases}. \]
	
	From here it can be verified by calculus that $F_2(z)>0$ for $z \in \Omega$. That means $F(z)$ is itself equal to the constant $C_1$ on $S_V,$ and $F(z)\geq C_1$ for all $z \notin S_V$. We conclude by Frostman's theorem that $\mu_V=\frac{1+q}{2\pi}\Delta Q dA \big{|}_{S_Q \backslash \Omega},$ and indeed $S_V=S_Q \backslash \Omega$.
	
	\begin{thm}
		\label{cavity}
		Let $Q(z)$ be an admissible external field and let the support of the equilibrium measure in the plane in the presence of $Q(z)$ be $S_Q$. Let $z_0$ be an interior point of $S_Q$ and assume that in a neighborhood of $z_0$, $Q(z)=C|\varphi(z)|^{2p}$ where $\varphi$ is a complex analytic function such that $\varphi(z_0)$ and $\varphi'(z_0)$ are nonzero, $p$ is a natural number, and $C>0$. Setting $r=\sqrt{q/2C(1+q)},$ we have that for all $q>0$ sufficiently small, the support $S_V$ of the equilibrium measure in the plane in the presence of the external field $V(z)=(1+q)Q(z)+q \ln |z-z_0|^{-1}$ is $S_V=S_Q \backslash \Omega$, where $\Omega:=\psi(D_r(\varphi^p(z_0)))$, with $D_r(\varphi^p(z_0))$ being the disc of radius $r$ centered at $\varphi^p(z_0)$, and $\psi$ being the (local) inverse of $\varphi^p$.
	\end{thm}
	{\textbf{Remark.}} In Theorem \ref{cavity} `sufficiently small' means small enough such that $\psi$ exists on $D_r(\varphi^p(z_0))$ and $\psi(D_r(\varphi^p(z_0))$ is compactly contained in $S_Q$.

	\section*{Acknowledgements}
	We thank Erik Lundberg for helpful remarks on Cassini ovals that allowed us to draw a connection to \cite{KhavinsonLundberg}.
	
	\bibliographystyle{plain}
	\bibliography{onepoint.bib}

\begin{thebibliography}{10}

\bibitem{AS}
D.~Aharanov and H.~Shapiro.
\newblock Domains on which analytic functions satisfy quadrature identities.
\newblock {\em Journal {d'Analyse} Math\'{e}matique}, 30:39--73, 1976.

\bibitem{Avci}
Y.~Avci.
\newblock {\em Quadrature Identities and the {Schwarz} Function}.
\newblock PhD thesis, Stanford University, 1977.

\bibitem{BH}
F.~Balogh and J.~Harnad.
\newblock Superharmonic perturbations of a {G}aussian measure, equilibrium
  measures, and orthogonal polynomials.
\newblock {\em Complex Analysis and Operator Theory}, 3:336--360, 2009.

\bibitem{Bell4}
S.~Bell.
\newblock The {Bergman} kernel and quadrature domains in the plane.
\newblock In Ebenfelt et~al. \cite{Proc}, pages 61--78.

\bibitem{Bell5}
S.~Bell.
\newblock Density of quadrature domains in one and several complex variables.
\newblock {\em Complex Variables and Elliptic Equations}, 54:165--171, 2009.

\bibitem{BellBook}
S.~Bell.
\newblock {\em The {Cauchy} Transform, Potential Theory and Conformal Mapping}.
\newblock CRC Press, Boca Raton, 2 edition, 2016.

\bibitem{Bergman}
S.~Bergman.
\newblock {\em The Kernel Function and Conformal Mapping}.
\newblock Number~V in Mathematical Surveys. American Mathematical Society,
  Providence, 1970.

\bibitem{Bers}
L.~Bers.
\newblock An approximation theorem.
\newblock {\em Journal d'Analyse Math{\'e}matique}, 14:1--4, 1965.

\bibitem{CK1}
J.~Criado~del Rey and Arno B.~J. Kuijlaars.
\newblock An equilibrium problem on the sphere with two equal charges.
\newblock {\em arXiv:1907.04801}, 2019.

\bibitem{CK2}
J.~Criado~del Rey and Arno B.~J. Kuijlaars.
\newblock A vector equilibrium problem for symmetrically located point charges
  on a sphere.
\newblock {\em Constructive Approximation}, To Appear.

\bibitem{CC}
Darren Crowdy and Martin Cloke.
\newblock Analytical solutions for distributed multipolar vortex equilibria on
  a sphere.
\newblock {\em Physics of Fluids}, 15(1):22--34, 2003.

\bibitem{Proc}
P.~Ebenfelt, B.~Gustafsson, D.~Khavinson, and M.~Putinar, editors.
\newblock {\em Quadrature Domains and Their Applications: The Harold S. Shapiro
  Anniversary Volume}, volume 156 of {\em Operator Theory and its
  Applications}. Birkh\"{a}user-Verlag, 2005.

\bibitem{ES}
B.~Epstein and M.~Schiffer.
\newblock On the mean-value property of harmonic functions.
\newblock {\em Journal d'Analyse Math\'{e}matique}, 14:109--111, 1965.

\bibitem{Gustafsson}
B.~Gustafsson.
\newblock Quadrature identities and the {Schottky} double.
\newblock {\em Acta Applicandae Mathematica}, 1(3):209--240, 1983.

\bibitem{GR}
B.~Gustafsson and J.~Roos.
\newblock Partial balayage on {R}iemannian manifolds.
\newblock {\em J. Math. Pures Appl.}, 118:82--127, 2018.

\bibitem{GS}
B.~Gustafsson and H.~Shapiro.
\newblock What is a quadrature domain?
\newblock In Ebenfelt et~al. \cite{Proc}, pages 1--25.

\bibitem{HJ}
P.~Haridas and J.~Janardhanan.
\newblock A 1-point poly-quadrature domain of order 1 not biholomorphic to a
  complete circular domain.
\newblock {\em Analysis and Mathematical Physics}, 9:1665--1668, 2019.

\bibitem{HM}
H.~Hedenmalm and N.~Makarov.
\newblock {C}oulomb ensembles and {L}aplacian growth.
\newblock {\em Proceedings of the London Mathematical Society},
  106(4):859--907, 2013.

\bibitem{KhavinsonLundberg}
D.~Khavinson and E.~Lundberg.
\newblock {\em Linear Holomorphic Partial Differential Equations and Classical
  Potential Theory}.
\newblock Number 232 in Mathematical Surveys and Monographs. American
  Mathematical Society, 2018.

\bibitem{KL}
Arno B.~J. Kuijlaars and A.~L{\'o}pez-Garc{\'i}a.
\newblock The normal matrix model with a monomial potential, a vector
  equilibrium problem, and multiple orthogonal polynomials on a star.
\newblock {\em Nonlinearity}, 28(2):347, 2015.

\bibitem{LD}
A.R. Legg and P.D. Dragnev.
\newblock Logarithmic equilibrium on the sphere in the presence of multiple
  point charges.
\newblock {\em Constructive Approximation}, 54(2):237--257, 2021.

\bibitem{Ransford}
Thomas Ransford.
\newblock {\em Potential theory in the complex plane}.
\newblock Cambridge University Press, 1995.

\bibitem{Roos}
J.~Roos.
\newblock Equilibrium measures and partial balayage.
\newblock {\em Complex Analysis and Operator Theory}, 9:65--85, 2015.

\bibitem{Rudin}
Walter Rudin.
\newblock {\em Real and complex analysis, 3rd edition}.
\newblock McGraw-Hill, 1987.

\bibitem{ST}
E.~Saff and V.~Totik.
\newblock {\em Logarithmic potentials with external fields}.
\newblock Number 316 in Grundlehren der mathematischen Wissenschaften.
  Springer-Verlag, 1997.

\bibitem{sakai}
M.~Sakai.
\newblock Regularity of a boundary having a $\text{Schwarz}$ function.
\newblock {\em Acta Math.}, 166:263--297, 1991.

\bibitem{Shapiro2}
H.~Shapiro.
\newblock {\em The {Schwarz} Function and its Generalization to Higher
  Dimensions}.
\newblock Wiley, New York, 1992.

\end{thebibliography}

\end{document}